\documentclass[reqno]{amsart}

\usepackage{header}
\definecolor{shadecolor}{cmyk}{0.3,0.1,0,0}

\title{Tree Dimension and the Sauer--Shelah Dichotomy}
\author{Roland Walker}

\begin{document}

\begin{abstract}
    We introduce tree dimension and its leveled variant in order to measure the complexity of leaf sets in binary trees. We then provide a tight upper bound on the size of such sets using leveled tree dimension. This, in turn, implies both the famous Sauer--Shelah Lemma for VC dimension and Bhaskar's version for Littlestone dimension, giving clearer insight into why these results place the exact same upper bound on their respective shatter functions. We also classify the isomorphism types of maximal leaf sets by tree dimension. Finally, we generalize this analysis to higher-arity trees.
\end{abstract}

\maketitle

\section{Binary Trees}\label{s:binary trees}
Let $T_n = (2^{\leq n}, \prec )$ be the \emph{binary tree} of height $n$. Its \emph{nodes} are the binary sequences of length at most $n$, and its \emph{ancestor relation} is given by 
$a\prec b$ if and only if $a$ is a proper initial segment of $b$. We call the empty sequence the \emph{root}, while sequences of length $n$ are called \emph{leaves}. Each leaf $b \in 2^n$ determines a \emph{branch} $\check{b}$ which contains all the initial segments of $b$ including $b$ itself. 
Given a set of leaves $B\subseteq 2^n$, we use $\check{B}$ to denote the corresponding set of branches $\{\check{b}:b\in B\}$. 

\begin{definition}
If $B\subseteq 2^n$ is a nonempty set of leaves, we define its \emph{tree dimension}, written $\TD(B)$, to be the largest nonnegative integer $d$ such that $T_d$ embeds in $(\check{B},\prec)$.
If $B$ is empty, we set $\TD(B)=-1$.
\end{definition}

In this paper, we use the model-theoretic notions of isomorphism and embedding. For a brief review, see the appendix.

Let $\widetilde{T}_n = (2^{\leq n}, \prec, \wedge, \sim)$ be the expansion of $T_n$ where the \emph{meet} function is interpreted so that $b \wedge b' = a$ if and only if $a$ is the largest initial segment of $b$ that is also an initial segment of $b'$ and 
the \emph{level} relation is interpreted so that $b \sim b'$ if and only if $|b|=|b'|$.
We call $\widetilde{T}_n$ the \emph{leveled binary tree} of height $n$. 

\begin{definition}
If $B\subseteq 2^n$ is a nonempty set of leaves, we define its \emph{leveled tree dimension}, written $\LTD(B)$, to be the largest nonnegative integer $d$ such that $\widetilde{T}_d$ embeds in $(\check{B},\prec, \wedge, \sim)$.
If $B$ is empty, we set $\LTD(B)=-1$. 
\end{definition}

Notice that for any $B\subseteq 2^n$, we have $\LTD(B)\leq \TD(B)$. This inequality can be strict; for example, if \[
B=\{000, 010, 100, 101\},
\]
then $\LTD(B) =1$ and $\TD(B) = 2$. See Figure~\ref{fig:strict ineq}.

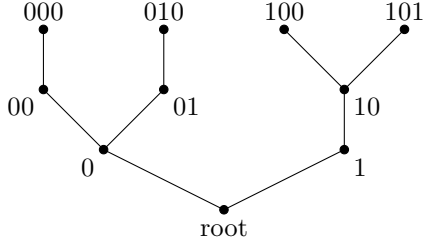
\begin{figure}[htbp]
\centering
\begin{tikzpicture}[scale=0.8]
    \draw (0,0) -- (2,1);
    \draw (0,0) -- (-2,1);
    \filldraw[black] (0,0) circle (2pt) node[anchor=north]{root};
    
    \draw (2,1) -- (2,2);
    \filldraw[black] (2,1) circle (2pt) node[anchor=north west]{1};
    \draw (-2,1) -- (-1,2);
    \filldraw[black] (-2,1) circle (2pt) node[anchor=north east]{0};

    \draw (-2,1) -- (-3,2);
    \filldraw[black] (-3,2) circle (2pt) node[anchor=north east]{00};
    \draw (-1,2) -- (-1,3);
    \filldraw[black] (-1,2) circle (2pt) node[anchor=north west]{01};
    \draw (2,2) -- (1,3);
    \filldraw[black] (2,2) circle (2pt) node[anchor=north west]{10};

    \draw (-3,2) -- (-3,3);
    \filldraw[black] (-3,3) circle (2pt) node[anchor=south]{000};
    \filldraw[black] (-1,3) circle (2pt) node[anchor=south]{010};
    \draw (2,2) -- (3,3);
    \filldraw[black] (3,3) circle (2pt) node[anchor=south]{101};
    \filldraw[black] (1,3) circle (2pt) node[anchor=south]{100};
\end{tikzpicture}
\caption{The branch diagram for a leaf set $B \subseteq 2^3$ with $\LTD(B) < \TD(B)$.}
\label{fig:strict ineq}
\end{figure}

\begin{theorem}\label{thm:ltd}
Given a set of leaves $B\subseteq 2^n$, if $\LTD(B) \leq d$, then 
\begin{equation}\label{eq:thm ltd}
|B| \leq \binom{n}{0} + \binom{n}{1}+ \cdots + \binom{n}{d}.
\end{equation}
\end{theorem}
\begin{proof}
The theorem clearly holds when $d = -1$ since, in this case, $B$ must be empty. It also holds when $d=n$ since
\[
\binom{n}{0}+\binom{n}{1}+\cdots+\binom{n}{n} =2^n.
\]

Assume the theorem holds for some $n \geq 0$, and let $B\subseteq 2^{n+1}$. Suppose $\LTD(B)\leq d$ with $0\leq d\leq n$. Let $A_1 = \{a\in 2^n: a\in \check{B}\}$.
Since $\LTD(A_1)\leq d$, our assumption implies that 
\begin{equation*} \label{eq:thm ltd A_1}
|A_1| \leq \binom{n}{0}+ \binom{n}{1}+\cdots+\binom{n}{d}.
\end{equation*}
Let $A_2 = \{a\in 2^n: \{a^{\frown}0, a^{\frown}1\}\subseteq B\}$.
We must have $\LTD(A_2)\leq d-1$, so our assumption implies that 
\begin{equation*}\label{eq:thm ltd A_2}
|A_2| \leq \binom{n}{0}+\binom{n}{1}+\cdots +\binom{n}{d-1}.
\end{equation*}
It follows that 
\[
|B| = |A_1|+|A_2|\leq \binom{n+1}{0}+\binom{n+1}{1}+\cdots +\binom{n+1}{d} 
\]
as desired.
\end{proof}

\subsection{Applications}

Fix a set $X$, and let $\mathcal{F}$ be a family of subsets of $X$. 
\begin{definition}
Given a finite $A\subseteq X$, we say $\mathcal{F}$ \emph{shatters} $A$ if $\{F\cap A: F\in \mathcal{F}\} = \mathcal{P}(A)$.
\end{definition}

\begin{definition}
If $\mathcal F$ is nonempty, then the \emph{Vapnik--Chervonenkis (VC) dimension}\footnote{Sauer refers to this quantity as ``density''  in \cite{sauer:densityoffamilies}, but it has become widely known as VC dimension since it appears in \cite[Theorem 1]{vapnikchervonenkis:uniformconvergence}, a slightly weaker precursor of the Sauer--Shelah Lemma.} of $\mathcal F$ is given by 
\[\VC(\mathcal{F}) = \sup\{|A| : \mathcal{F} \text{ shatters } A\}.\]
If $\mathcal F$ is empty, we set $\VC(\mathcal F) = -1$.
\end{definition}

\begin{definition}
Given $\bar{a} = (a_0, \ldots, a_{n-1}) \in X^n$, we define the characteristic function $\chi_{\bar{a}} : \mathcal{F} \to 2^n$ so that \[\chi_{\bar{a}}(F)(i) = \begin{cases}
    1 & \text{if } a_i \in F, \\
    0 & \text{otherwise}.
\end{cases}\]
\end{definition}

Notice that for all $\bar{a}\in X^n$, we have $\LTD(\chi_{\bar{a}}(\mathcal{F})) \leq \VC(\mathcal{F})$, so Theorem~\ref{thm:ltd} immediately yields the original Sauer--Shelah Lemma \cite[Theorem 2]{sauer:densityoffamilies} as a corollary.

\begin{cor}[Sauer--Shelah Lemma]
Given $A\subseteq X$, if $|A|\leq n$ and $\VC(\mathcal{F}) \leq d$, then 
\begin{equation}\label{eq:ss}
|\{F\cap A: F\in \mathcal{F}\}| \leq \binom{n}{0} + \binom{n}{1}+ \cdots + \binom{n}{d}.
\end{equation}
\end{cor}

Many proofs of this famous result have been recorded in the literature. Theorem~\ref{thm:ltd} not only provides a novel, straightforward proof of the original Sauer--Shelah Lemma for bounded VC dimension but also unifies it with the analogous result for bounded Littlestone dimension (see Corollary~\ref{cor:stable ss}). This gives clear insight into why the right hand sides of (\ref{eq:ss}) and (\ref{eq:stable ss}) are identical.

\begin{definition}
Given a labeling $\alpha: 2^{<n}\to X$, we define the characteristic function $\chi_{\alpha}:\mathcal{F} \to 2^n$ so that for each $s \prec \chi_{\alpha}(F)$, the sequence $s^{\frown}1$ is an initial segment of $\chi_{\alpha}(F)$ if and only if $\alpha(s)\in F$.
\end{definition}

\begin{definition}
If $\mathcal{F}$ is nonempty, then the \emph{Littlestone dimension}\footnote{Bhaskar calls this quantity ``thicket dimension'' in \cite[Definition 3.6]{bhaskar:thicketdensity} but notes its other names in \cite[Remark 3.7]{bhaskar:thicketdensity}. See also \cite[Definition 8]{bendavid:agnosticonlinelearning} and \cite{littlestone:learningquickly}.} of $\mathcal{F}$ is given by
\[\LD(\mathcal{F}) = \sup\{d : \chi_\alpha(\mathcal{F})=2^{d} \text{ for some labeling } \alpha : 2^{<d}\rightarrow X\}.\]
If $\mathcal{F}$ is empty, we set $\LD(\mathcal{F})=-1$.
\end{definition}

Notice that for any labeling $\alpha : 2^{<n} \rightarrow X$, we have \[\LTD(\chi_\alpha(\mathcal{F})) \leq \TD(\chi_\alpha(\mathcal{F})) \leq \LD(\mathcal{F}),\] so Theorem~\ref{thm:ltd} yields Bhaskar's so-called ``thicket'' version of the Sauer--Shelah Lemma \cite[Theorem 4.3]{bhaskar:thicketdensity} as a corollary.

\begin{cor}\label{cor:stable ss}
Given a labeling $\alpha : 2^{<n} \rightarrow X$, if $\LD(\mathcal{F}) \leq d$, then 
\begin{equation}\label{eq:stable ss}
|\chi_\alpha(\mathcal{F})| \leq \binom{n}{0} + \binom{n}{1} + \cdots + \binom{n}{d}.
\end{equation}
\end{cor}

We can also apply Theorem~\ref{thm:ltd} to banned binary sequence problems (BBSP). Due to the technical nature of the definition of a hereditary BBSP we will not restate it here. Rather, we refer the interested reader to \cite[Definitions 3.1 and 3.6]{chase:bbsp}. Notice that if $B\subseteq 2^n$ is a set of solutions to a hereditary $(d+1)$-fold BBSP of length $n$, then $\LTD(B)\leq d$, so Theorem~\ref{thm:ltd} yields \cite[Theorem 3.7]{chase:bbsp} as a corollary. 

\begin{cor}
A hereditary $(d+1)$-fold BBSP of height $n$ has at most
\[\binom{n}{0} + \binom{n}{1} +\cdots + \binom{n}{d}\]
solutions. 
\end{cor}

\subsection{Maximal Leaf Sets}

We say $B\subseteq 2^n$ is a \emph{maximal leaf set} with tree dimension $d$, and write $\TD(B) = d^+$, when $\TD(B) = d$ but $\TD(B\cup \{b\})> d$ for all $b\in 2^n\setminus B$. We follow the same convention for leveled tree dimension.

\begin{definition}
Given a node $a\in 2^{\leq n}$, we define its \emph{norm} by
\[ \|a\| = \sum_{i<|a|} a(i). \]
\end{definition}

\begin{definition}
If $A$ is a nonempty set of nodes, we define its \emph{norm} by
\[\|A\| = \max_{a\in A} \|a\|.\] 
If $A$ is empty, we set $\|A\| =-1$. 
\end{definition}

Notice that for any leaf $b$, we have $\|\check{b}\| = \|b\|$. Similarly, if $B$ is a set of leaves, then $\|\check{B}\| = \|B\|$.

It is easy to see that the upper bound provided by (\ref{eq:thm ltd}) is tight since $\LTD(B) = d$ when $B = \{b\in 2^n: \|b\|\leq d\}$. Based on this, one might conjecture that (\ref{eq:thm ltd}) becomes an equality for maximal leaf sets. This is true when $n\leq 5$; however, Figure~\ref{fig:max set exception} illustrates a set of leaves $B \subseteq 2^{6}$ with $\LTD(B)= 2^+$ but
\[|B| \ = \ 21 \ < \ 22 \ = \ \binom{6}{0} + \binom{6}{1} + \binom{6}{2}.\]
Although the conjecture fails for leveled tree dimension, we will show that it holds for tree dimension (see Corollary~\ref{cor:isotp max td}). Thus, for the leaf set $B$ pictured in Figure~\ref{fig:max set exception}, we must have $\LTD(B) < \TD(B)$. Indeed, if we look at the shaded portion of the diagram, it becomes clear that $\TD(B) \geq 4$.

\begin{figure}[h]
\centering
\begin{tikzpicture}[scale=0.1]

    % \filldraw[shadecolor] (-61,60) circle (20pt);
    % \draw[line width=1.2mm, shadecolor ] (-61,60) -- (-60,50) ;
    % \filldraw[shadecolor] (-59,60) circle (20pt);
    % \draw [line width=1.2mm, shadecolor ] (-59, 60) -- (-60,50);

    \filldraw[shadecolor] (-61,60) circle (20pt);
    \draw[line width=1.2mm, shadecolor] (-61,60) -- (-60,50);
    \filldraw[shadecolor] (-59,60) circle (20pt);
    \draw[line width=1.2mm, shadecolor] (-59, 60) -- (-60,50);
    
    \filldraw[shadecolor] (-53, 60) circle (20pt);
    \draw[line width=1.2mm, shadecolor] (-53, 60) -- (-52,50);
    \filldraw[shadecolor] (-51,60) circle (20pt);
    \draw[line width=1.2mm, shadecolor] (-51, 60) -- (-52,50);

    \filldraw[shadecolor] (-45, 60) circle (20pt);
    \draw[line width=1.2mm, shadecolor] (-45, 60) -- (-44,50);
    \filldraw[shadecolor] (-43,60) circle (20pt);
    \draw[line width=1.2mm, shadecolor] (-43, 60) -- (-44,50);

    % \filldraw[shadecolor] (-38, 60) circle (20pt);
    % \draw[line width=1.2mm, shadecolor] (-38, 60) -- (-38,50);
    % \filldraw[shadecolor] (-34,60) circle (20pt);
    % \draw[line width=1.2mm, shadecolor] (-34, 60) -- (-34,50);

    %  \filldraw[shadecolor] (-17,60) circle (20pt);
    % \draw[line width=1.2mm, shadecolor] (-17,60) -- (-16,50);
    % \filldraw[shadecolor] (-15,60) circle (20pt);
    % \draw[line width=1.2mm, shadecolor] (-15, 60) -- (-16,50);

    \filldraw[shadecolor] (7, 60) circle (20pt);
    \draw[line width=1.2mm, shadecolor] (7, 60) -- (8,50);
    \filldraw[shadecolor] (9,60) circle (20pt);
    \draw[line width=1.2mm, shadecolor] (9, 60) -- (8,50);

    % \filldraw[shadecolor] (18, 60) circle (20pt);
    % \draw[line width=1.2mm, shadecolor] (18, 60) -- (18,50);
    % \filldraw[shadecolor] (22,60) circle (20pt);
    % \draw[line width=1.2mm, shadecolor] (22, 60) -- (22,50);

    % \filldraw[shadecolor] (28, 60) circle (20pt);
    % \draw[line width=1.2mm, shadecolor] (28, 60) -- (28,50);
   
    % \filldraw[shadecolor] (7, 60) circle (20pt);
    % \draw[line width=1.2mm, shadecolor] (7, 60) -- (8,50);
    % \filldraw[shadecolor] (9,60) circle (20pt);
    % \draw[line width=1.2mm, shadecolor] (9, 60) -- (8,50);

    % \filldraw[shadecolor] (34, 60) circle (20pt);
    % \draw[line width=1.2mm, shadecolor] (34, 60) -- (34,50);
    % \filldraw[shadecolor] (38,60) circle (20pt);
    % \draw[line width=1.2mm, shadecolor] (38, 60) -- (38,50);

    % \filldraw[shadecolor] (42, 60) circle (20pt);
    % \draw[line width=1.2mm, shadecolor] (42, 60) -- (42,50);
    % \filldraw[shadecolor] (46,60) circle (20pt);
    % \draw[line width=1.2mm, shadecolor] (46, 60) -- (46,50);

    % \filldraw[shadecolor] (54, 60) circle (20pt);
    % \draw[line width=1.2mm, shadecolor] (54, 60) -- (54,50);
    % \filldraw[shadecolor] (58,60) circle (20pt);
    % \draw[line width=1.2mm, shadecolor] (58, 60) -- (58,50);
    
    %------------------------------
    \filldraw[shadecolor] (-60, 50) circle (20pt);
    \draw[line width=1.2mm, shadecolor] (-60, 50) -- (-60,40);
    \filldraw[shadecolor] (-52,50) circle (20pt);
    \draw[line width=1.2mm, shadecolor] (-52, 50) -- (-52,40);
    
    \filldraw[shadecolor] (-44, 50) circle (20pt);
    \draw[line width=1.2mm, shadecolor] (-44, 50) -- (-44,40);
    \filldraw[shadecolor] (-38,50) circle (20pt);
    \draw[line width=1.2mm, shadecolor] (-38, 50) -- (-36,40);
    \filldraw[shadecolor] (-34,50) circle (20pt);
    \draw[line width=1.2mm, shadecolor] (-34, 50) -- (-36,40);
    
    % \filldraw[shadecolor] (-16, 50) circle (20pt);
    % \draw[line width=1.2mm, shadecolor] (-16, 50) -- (-16,40);
    
    \filldraw[shadecolor] (8, 50) circle (20pt);
    \draw[line width=1.2mm, shadecolor] (8, 50) -- (8,40);
    
    % \filldraw[shadecolor] (18, 50) circle (20pt);
    % \draw[line width=1.2mm, shadecolor] (18, 50) -- (20,40);
    % \filldraw[shadecolor] (22,50) circle (20pt);
    % \draw[line width=1.2mm, shadecolor] (22, 50) -- (20,40);
    
    % \filldraw[shadecolor] (28, 50) circle (20pt);
    % \draw[line width=1.2mm, shadecolor] (28, 50) -- (28,40);
    
    % \filldraw[shadecolor] (34, 50) circle (20pt);
    % \draw[line width=1.2mm, shadecolor] (34, 50) -- (36,40);
    % \filldraw[shadecolor] (38,50) circle (20pt);
    % \draw[line width=1.2mm, shadecolor] (38, 50) -- (36,40);
    
    % \filldraw[shadecolor] (42, 50) circle (20pt);
    % \draw[line width=1.2mm, shadecolor] (42, 50) -- (44,40);
    % \filldraw[shadecolor] (46,50) circle (20pt);
    % \draw[line width=1.2mm, shadecolor] (46, 50) -- (44,40);
    
    \filldraw[shadecolor] (54, 50) circle (20pt);
    \draw[line width=1.2mm, shadecolor] (54, 50) -- (56,40);
    \filldraw[shadecolor] (58,50) circle (20pt);
    \draw[line width=1.2mm, shadecolor] (58, 50) -- (56,40);
    
    %------------------------
     \filldraw[shadecolor] (-60, 40) circle (20pt);
    \draw[line width=1.2mm, shadecolor] (-60, 40) -- (-56,30);
    \filldraw[shadecolor] (-52,40) circle (20pt);
    \draw[line width=1.2mm, shadecolor] (-52, 40) -- (-56,30);
    
    \filldraw[shadecolor] (-44, 40) circle (20pt);
    \draw[line width=1.2mm, shadecolor] (-44, 40) -- (-40,30);
    \filldraw[shadecolor] (-36,40) circle (20pt);
    \draw[line width=1.2mm, shadecolor] (-36, 40) -- (-40,30);
    
    % \filldraw[shadecolor] (-16,40) circle (20pt);
    % \draw[line width=1.2mm, shadecolor] (-16, 40) -- (-16,30);
    
    \filldraw[shadecolor] (8,40) circle (20pt);
    \draw[line width=1.2mm, shadecolor] (8, 40) -- (8,30);
    
    \filldraw[shadecolor] (20, 40) circle (20pt);
    \draw[line width=1.2mm, shadecolor] (20, 40) -- (24,30);
    \filldraw[shadecolor] (28,40) circle (20pt);
    \draw[line width=1.2mm, shadecolor] (28, 40) -- (24,30);
    
    \filldraw[shadecolor] (36, 40) circle (20pt);
    \draw[line width=1.2mm, shadecolor] (36, 40) -- (40,30);
    \filldraw[shadecolor] (44,40) circle (20pt);
    \draw[line width=1.2mm, shadecolor] (44, 40) -- (40,30);
    
    \filldraw[shadecolor] (56,40) circle (20pt);
    \draw[line width=1.2mm, shadecolor] (56, 40) -- (56,30);
    
    %-----------------------------------
    
    \filldraw[shadecolor] (-56, 30) circle (20pt);
    \draw[line width=1.2mm, shadecolor] (-56, 30) -- (-48,20);
    \filldraw[shadecolor] (-40,30) circle (20pt);
    \draw[line width=1.2mm, shadecolor] (-40, 30) -- (-48,20);
     
    % \filldraw[shadecolor] (-16,30) circle (20pt);
    % \draw[line width=1.2mm, shadecolor] (-16, 30) -- (-16,20);
    
    \filldraw[shadecolor] (8, 30) circle (20pt);
    \draw[line width=1.2mm, shadecolor] (8, 30) -- (16,20);
    \filldraw[shadecolor] (24,30) circle (20pt);
    \draw[line width=1.2mm, shadecolor] (24, 30) -- (16,20);
    
    \filldraw[shadecolor] (40, 30) circle (20pt);
    \draw[line width=1.2mm, shadecolor] (40, 30) -- (48,20);
    \filldraw[shadecolor] (56,30) circle (20pt);
    \draw[line width=1.2mm, shadecolor] (56, 30) -- (48,20);
    
    %-------------------------------
    
    \filldraw[shadecolor] (-48, 20) circle (20pt);
    \draw[line width=1.2mm, shadecolor] (-48, 20) -- (-32,10);
    % \filldraw[shadecolor] (-16,20) circle (20pt);
    % \draw[line width=1.2mm, shadecolor] (-16, 20) -- (-32,10);
    
    \filldraw[shadecolor] (16, 20) circle (20pt);
    \draw[line width=1.2mm, shadecolor] (16, 20) -- (32,10);
    \filldraw[shadecolor] (48,20) circle (20pt);
    \draw[line width=1.2mm, shadecolor] (48, 20) -- (32,10);
    
    \filldraw[shadecolor] (-32, 10) circle (20pt);
    \draw[line width=1.2mm, shadecolor] (-32, 10) -- (0,0);
    \filldraw[shadecolor] (32,10) circle (20pt);
    \draw[line width=1.2mm, shadecolor] (32, 10) -- (0,0);
    \filldraw[shadecolor] (0,0) circle (20pt);

%=========================
% black graph

\filldraw[black] (-61,60) circle (10pt);
    \draw[line width = 0.20mm] (-61,60) -- (-60,50);
    \filldraw[black] (-59,60) circle (10pt);
    \draw[line width = 0.20mm] (-59, 60) -- (-60,50);

    \filldraw[black] (-53, 60) circle (10pt);
    \draw[line width = 0.20mm] (-53, 60) -- (-52,50);
    \filldraw[black] (-51,60) circle (10pt);
    \draw[line width = 0.20mm] (-51, 60) -- (-52,50);

    \filldraw[black] (-45, 60) circle (10pt);
    \draw[line width = 0.20mm] (-45, 60) -- (-44,50);
    \filldraw[black] (-43,60) circle (10pt);
    \draw[line width = 0.20mm] (-43, 60) -- (-44,50);

    \filldraw[black] (-38, 60) circle (10pt);
    \draw[line width = 0.20mm] (-38, 60) -- (-38,50);
    \filldraw[black] (-34,60) circle (10pt);
    \draw[line width = 0.20mm] (-34, 60) -- (-34,50);

     \filldraw[black] (-17,60) circle (10pt);
    \draw[line width = 0.20mm] (-17,60) -- (-16,50);
    \filldraw[black] (-15,60) circle (10pt);
    \draw[line width = 0.20mm] (-15, 60) -- (-16,50);

    \filldraw[black] (7, 60) circle (10pt);
    \draw[line width = 0.20mm] (7, 60) -- (8,50);
    \filldraw[black] (9,60) circle (10pt);
    \draw[line width = 0.20mm] (9, 60) -- (8,50);

    \filldraw[black] (18, 60) circle (10pt);
    \draw[line width = 0.20mm] (18, 60) -- (18,50);
    \filldraw[black] (22,60) circle (10pt);
    \draw[line width = 0.20mm] (22, 60) -- (22,50);

    \filldraw[black] (28, 60) circle (10pt);
    \draw[line width = 0.20mm] (28, 60) -- (28,50);
   
    \filldraw[black] (7, 60) circle (10pt);
    \draw[line width = 0.20mm] (7, 60) -- (8,50);
    \filldraw[black] (9,60) circle (10pt);
    \draw[line width = 0.20mm] (9, 60) -- (8,50);

    \filldraw[black] (34, 60) circle (10pt);
    \draw[line width = 0.20mm] (34, 60) -- (34,50);
    \filldraw[black] (38,60) circle (10pt);
    \draw[line width = 0.20mm] (38, 60) -- (38,50);

    \filldraw[black] (42, 60) circle (10pt);
    \draw[line width = 0.20mm] (42, 60) -- (42,50);
    \filldraw[black] (46,60) circle (10pt);
    \draw[line width = 0.20mm] (46, 60) -- (46,50);

    \filldraw[black] (54, 60) circle (10pt);
    \draw[line width = 0.20mm] (54, 60) -- (54,50);
    \filldraw[black] (58,60) circle (10pt);
    \draw[line width = 0.20mm] (58, 60) -- (58,50);
    
    %------------------------------
    \filldraw[black] (-60, 50) circle (10pt);
    \draw[line width = 0.20mm] (-60, 50) -- (-60,40);
    \filldraw[black] (-52,50) circle (10pt);
    \draw[line width = 0.20mm] (-52, 50) -- (-52,40);
    
    \filldraw[black] (-44, 50) circle (10pt);
    \draw[line width = 0.20mm] (-44, 50) -- (-44,40);
    \filldraw[black] (-38,50) circle (10pt);
    \draw[line width = 0.20mm] (-38, 50) -- (-36,40);
    \filldraw[black] (-34,50) circle (10pt);
    \draw[line width = 0.20mm] (-34, 50) -- (-36,40);
    
    \filldraw[black] (-16, 50) circle (10pt);
    \draw[line width = 0.20mm] (-16, 50) -- (-16,40);
    
    \filldraw[black] (8, 50) circle (10pt);
    \draw[line width = 0.20mm] (8, 50) -- (8,40);
    
    \filldraw[black] (18, 50) circle (10pt);
    \draw[line width = 0.20mm] (18, 50) -- (20,40);
    \filldraw[black] (22,50) circle (10pt);
    \draw[line width = 0.20mm] (22, 50) -- (20,40);
    
    \filldraw[black] (28, 50) circle (10pt);
    \draw[line width = 0.20mm] (28, 50) -- (28,40);
    
    \filldraw[black] (34, 50) circle (10pt);
    \draw[line width = 0.20mm] (34, 50) -- (36,40);
    \filldraw[black] (38,50) circle (10pt);
    \draw[line width = 0.20mm] (38, 50) -- (36,40);
    
    \filldraw[black] (42, 50) circle (10pt);
    \draw[line width = 0.20mm] (42, 50) -- (44,40);
    \filldraw[black] (46,50) circle (10pt);
    \draw[line width = 0.20mm] (46, 50) -- (44,40);
    
    \filldraw[black] (54, 50) circle (10pt);
    \draw[line width = 0.20mm] (54, 50) -- (56,40);
    \filldraw[black] (58,50) circle (10pt);
    \draw[line width = 0.20mm] (58, 50) -- (56,40);
    
    %------------------------
     \filldraw[black] (-60, 40) circle (10pt);
    \draw[line width = 0.20mm] (-60, 40) -- (-56,30);
    \filldraw[black] (-52,40) circle (10pt);
    \draw[line width = 0.20mm] (-52, 40) -- (-56,30);
    
    \filldraw[black] (-44, 40) circle (10pt);
    \draw[line width = 0.20mm] (-44, 40) -- (-40,30);
    \filldraw[black] (-36,40) circle (10pt);
    \draw[line width = 0.20mm] (-36, 40) -- (-40,30);
    
    \filldraw[black] (-16,40) circle (10pt);
    \draw[line width = 0.20mm] (-16, 40) -- (-16,30);
    
    \filldraw[black] (8,40) circle (10pt);
    \draw[line width = 0.20mm] (8, 40) -- (8,30);
    
    \filldraw[black] (20, 40) circle (10pt);
    \draw[line width = 0.20mm] (20, 40) -- (24,30);
    \filldraw[black] (28,40) circle (10pt);
    \draw[line width = 0.20mm] (28, 40) -- (24,30);
    
    \filldraw[black] (36, 40) circle (10pt);
    \draw[line width = 0.20mm] (36, 40) -- (40,30);
    \filldraw[black] (44,40) circle (10pt);
    \draw[line width = 0.20mm] (44, 40) -- (40,30);
    
    \filldraw[black] (56,40) circle (10pt);
    \draw[line width = 0.20mm] (56, 40) -- (56,30);
    
    \filldraw[black] (-56, 30) circle (10pt);
    \draw[line width = 0.20mm] (-56, 30) -- (-48,20);
    \filldraw[black] (-40,30) circle (10pt);
    \draw[line width = 0.20mm] (-40, 30) -- (-48,20);
     
    \filldraw[black] (-16,30) circle (10pt);
    \draw[line width = 0.20mm] (-16, 30) -- (-16,20);
    
    \filldraw[black] (8, 30) circle (10pt);
    \draw[line width = 0.20mm] (8, 30) -- (16,20);
    \filldraw[black] (24,30) circle (10pt);
    \draw[line width = 0.20mm] (24, 30) -- (16,20);
    
    \filldraw[black] (40, 30) circle (10pt);
    \draw[line width = 0.20mm] (40, 30) -- (48,20);
    \filldraw[black] (56,30) circle (10pt);
    \draw[line width = 0.20mm] (56, 30) -- (48,20);
    
    \filldraw[black] (-48, 20) circle (10pt);
    \draw[line width = 0.20mm] (-48, 20) -- (-32,10);
    \filldraw[black] (-16,20) circle (10pt);
    \draw[line width = 0.20mm] (-16, 20) -- (-32,10);
    
    \filldraw[black] (16, 20) circle (10pt);
    \draw[line width = 0.20mm] (16, 20) -- (32,10);
    \filldraw[black] (48,20) circle (10pt);
    \draw[line width = 0.20mm] (48, 20) -- (32,10);
    
    \filldraw[black] (-32, 10) circle (10pt);
    \draw[line width = 0.20mm] (-32, 10) -- (0,0);
    \filldraw[black] (32,10) circle (10pt);
    \draw[line width = 0.20mm] (32, 10) -- (0,0);
    \filldraw[black] (0,0) circle (10pt);

\end{tikzpicture}
\caption{The branch diagram of a leaf set $B \subseteq 2^6$ with $\LTD(B) = 2^+$ and $\TD(B) = 4$.}
\label{fig:max set exception}
\end{figure}
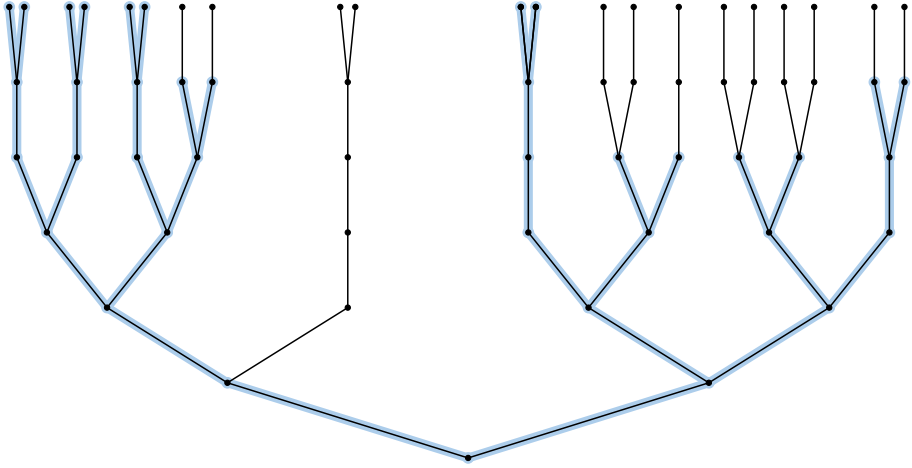

\begin{definition}For each $a \in 2^{< n}$, let $\sigma_a: T_n\rightarrow T_n $ be the automorphism mapping 
\[a ^\frown 0^\frown b \ \mapsto \ a^\frown 1 ^\frown b \quad \qquad \text{ and } \quad \qquad a ^\frown 1^\frown b \ \mapsto \ a^\frown 0 ^\frown b\]
for all sequences $b$ of length at most $n-|a|-1$, leaving all other nodes fixed. 
\end{definition}

\begin{definition}
Given a node $a$ and a set of nodes $A$, we use $A\dhr_a$ to denote the restriction of $A$ to nodes with initial segment $a$. 
\end{definition}

\begin{definition}
Given a set of leaves $B\subseteq 2^n$, we define its \emph{normalization} $\ddot{B}$ using the following inductive procedure: Let $N=2^{n-1}+2^{n-2} +\cdots + 2^0$, and let $(a_i:i<N)$ enumerate the elements of $2^{<n}$ in such a way that if $a_i$ is longer than $a_j$, then $i < j.$
Let $B_0 = B.$ For each $i<N$, once $B_i$ has been defined, let 
\[B_{i+1} = 
    \begin{cases}
        B_{i} & \text{ if } \|B_{i}\dhr_{a_i}\| \leq \|\sigma_{a_{i}}(B_{i}\dhr_{a_i})\|, \\
        \sigma_{a_{i}}(B_{i}) & \text{ otherwise.}
    \end{cases}
\]
Let $\ddot{B}=B_N$.
\end{definition}

Notice that for any leaf set $B$, we have $(\check{B}, \prec) \cong (\check{\ddot{B}}, \prec)$.

\begin{theorem}
\label{lem:stabless}
If $B \subseteq 2^n$ is a set of leaves, then $\TD(B) = \|\ddot{B}\|$. 
\end{theorem}

\begin{proof}
Clearly, we have $\TD(B) \leq \|\ddot{B}\|$ since $T_{\|\ddot{B}\|+1}$ cannot embed in \[
(\{a\in 2^{\leq n}: \|a\|\leq \|\ddot{B}\|\},\prec).
\] Thus, it is enough to show that  $\TD(B)\geq d$ whenever $\|\ddot{B}\|\geq d$. This is obviously true when $d\leq 0$. 
Assume it holds for some $d\geq 0$, and suppose we have $b\in \ddot{B}$ with $\|b\| \geq d+1$. Let $a$ be the largest initial segment of $b$ such that $\|a\|=0$. By our assumption, it follows that $T_d$ embeds in $(\check{\ddot{B}}\dhr_{a^\frown 0},\prec)$ and $(\check{\ddot{B}}\dhr_{a^\frown 1}, \prec)$.
\end{proof}

\begin{cor}\label{cor:isotp max td}
Given a set of leaves $B\subseteq 2^n$, if $\TD(B) = d^+$, then \[(\check{B},\prec) \cong (\{a\in 2^{\leq n} : \|a\| \leq d\},\prec).\]
\end{cor}

\section{Higher-Arity Trees}
For $m\geq 2$, let 
\begin{itemize}
	\item[] $T_{n,m} = (m^{\leq n}, \prec)$ be the \emph{$m$-ary tree},
	\item[] $\hat{T}_{n,m} = (m^{\leq n}, \prec, \wedge)$ be the \emph{meeted $m$-ary tree}, and
	\item[] $\widetilde{T}_{n,m} = (m^{\leq n}, \prec, \wedge, \sim)$ be the \emph{leveled $m$-ary tree},
\end{itemize}
all of height $n$.
Our nodes are now the $m$-ary sequences of length at most $n$, while the symbols $\prec$, $\wedge$, and $\sim$ are interpreted exactly as defined in Section~\ref{s:binary trees} for binary trees.

\begin{definition}
For $m \geq \ell \geq 2$, if $B\subseteq m^n$ is a nonempty set of leaves, we define its \emph{$\ell$-ary tree dimension}, written $\TD_{\ell}(B)$, to be the largest nonnegative integer $d$ such that $T_{d,\ell}$ embeds in $(\check{B},\prec)$. 
Likewise, we define its \emph{meeted $\ell$-ary tree dimension}, written $\MTD_{\ell}(B)$, to be the largest nonnegative integer $d$ such that $\hat{T}_{d,\ell}$ embeds in $(\check{B},\prec, \wedge)$ and
its \emph{leveled $\ell$-ary tree dimension}, written $\LTD_{\ell}(B)$, to be the largest nonnegative integer $d$ such that $\widetilde{T}_{d,\ell}$ embeds in $(\check{B},\prec, \wedge, \sim)$.
If $B$ is empty, we set 
\[
\TD_{\ell}(B) = \MTD_{\ell}(B) = \LTD_{\ell}(B) = -1. 
\]
\end{definition}

Notice that for any $B\subseteq m^n$, we have
\[
\LTD_{\ell}(B) \leq \MTD_{\ell}(B) \leq \TD_{\ell}(B).
\]
For $\ell \geq 3$, both inequalities may be strict. We give an example of this in Figure \ref{fig:strict ineq ha}. However, for $\ell = 2$, we have $\MTD_2(B) = \TD_2(B)$.
In fact, if $B\subseteq 2^n$, then
\[
\LTD(B) = \LTD_2(B) \leq \MTD_2(B) = \TD_2(B) = \TD(B).
\]

\begin{figure}[h]
\centering
\begin{tikzpicture}[scale=0.09]

%=========================
% black graph

    \filldraw[black] (0,0) circle (10pt);

    \filldraw[black] (-45,10) circle (10pt);
    \draw[line width = 0.20mm] (0,0) -- (-45,10);
    \filldraw[black] (0,10) circle (10pt);
    \draw[line width = 0.20mm] (0,0) -- (0,10);
    \filldraw[black] (45,10) circle (10pt);
    \draw[line width = 0.20mm] (0,0) -- (45,10);
    
    \filldraw[black] (-45,20) circle (10pt);
    \draw[line width = 0.20mm] (-45,10) -- (-45,20);

    \filldraw[black] (-15,20) circle (10pt);
    \draw[line width = 0.20mm] (0,10) -- (-15,20);
    \filldraw[black] (0,20) circle (10pt);
    \draw[line width = 0.20mm] (0,10) -- (0,20);
    \filldraw[black] (15,20) circle (10pt);
    \draw[line width = 0.20mm] (0,10) -- (15,20);
    
    \filldraw[black] (45,20) circle (10pt);
    \draw[line width = 0.20mm] (45,10) -- (45,20);

    \filldraw[black] (-60,30) circle (10pt);
    \draw[line width = 0.20mm] (-45,20) -- (-60,30);
    \filldraw[black] (-45,30) circle (10pt);
    \draw[line width = 0.20mm] (-45,20) -- (-45,30);
    \filldraw[black] (-30,30) circle (10pt);
    \draw[line width = 0.20mm] (-45,20) -- (-30,30);

    \filldraw[black] (-15,30) circle (10pt);
    \draw[line width = 0.20mm] (-15,20) -- (-15,30);
    \filldraw[black] (0,30) circle (10pt);
    \draw[line width = 0.20mm] (0,20) -- (0,30);
    \filldraw[black] (15,30) circle (10pt);
    \draw[line width = 0.20mm] (15,20) -- (15,30);

    \filldraw[black] (60,30) circle (10pt);
    \draw[line width = 0.20mm] (45,20) -- (60,30);
    \filldraw[black] (45,30) circle (10pt);
    \draw[line width = 0.20mm] (45,20) -- (45,30);
    \filldraw[black] (30,30) circle (10pt);
    \draw[line width = 0.20mm] (45,20) -- (30,30);
    
%%%%%%%%%%%%%%%%%%%%%%%%%%%%%%

    \filldraw[black] (-62.5,40) circle (10pt);
    \draw[line width = 0.20mm] (-60,30) -- (-62.5,40);
    \filldraw[black] (-65,50) circle (10pt);
    \draw[line width = 0.20mm] (-62.5,40) -- (-65,50);
    \filldraw[black] (-57.5,40) circle (10pt);
    \draw[line width = 0.20mm] (-60,30) -- (-57.5,40);
    \filldraw[black] (-60,50) circle (10pt);
    \draw[line width = 0.20mm] (-57.5,40) -- (-60,50);
    \filldraw[black] (-55,50) circle (10pt);
    \draw[line width = 0.20mm] (-57.5,40) -- (-55,50);

    \filldraw[black] (-47.5,40) circle (10pt);
    \draw[line width = 0.20mm] (-45,30) -- (-47.5,40);
    \filldraw[black] (-50,50) circle (10pt);
    \draw[line width = 0.20mm] (-47.5,40) -- (-50,50);
    \filldraw[black] (-42.5,40) circle (10pt);
    \draw[line width = 0.20mm] (-45,30) -- (-42.5,40);
    \filldraw[black] (-45,50) circle (10pt);
    \draw[line width = 0.20mm] (-42.5,40) -- (-45,50);
    \filldraw[black] (-40,50) circle (10pt);
    \draw[line width = 0.20mm] (-42.5,40) -- (-40,50);
    
    \filldraw[black] (-32.5,40) circle (10pt);
    \draw[line width = 0.20mm] (-30,30) -- (-32.5,40);
    \filldraw[black] (-35,50) circle (10pt);
    \draw[line width = 0.20mm] (-32.5,40) -- (-35,50);
    \filldraw[black] (-27.5,40) circle (10pt);
    \draw[line width = 0.20mm] (-30,30) -- (-27.5,40);
    \filldraw[black] (-30,50) circle (10pt);
    \draw[line width = 0.20mm] (-27.5,40) -- (-30,50);
    \filldraw[black] (-25,50) circle (10pt);
    \draw[line width = 0.20mm] (-27.5,40) -- (-25,50);

    \filldraw[black] (-17.5,40) circle (10pt);
    \draw[line width = 0.20mm] (-15,30) -- (-17.5,40);
    \filldraw[black] (-20,50) circle (10pt);
    \draw[line width = 0.20mm] (-17.5,40) -- (-20,50);
    \filldraw[black] (-12.5,40) circle (10pt);
    \draw[line width = 0.20mm] (-15,30) -- (-12.5,40);
    \filldraw[black] (-15,50) circle (10pt);
    \draw[line width = 0.20mm] (-12.5,40) -- (-15,50);
    \filldraw[black] (-10,50) circle (10pt);
    \draw[line width = 0.20mm] (-12.5,40) -- (-10,50);

    \filldraw[black] (-2.5,40) circle (10pt);
    \draw[line width = 0.20mm] (0,30) -- (-2.5,40);
    \filldraw[black] (-5,50) circle (10pt);
    \draw[line width = 0.20mm] (-2.5,40) -- (-5,50);
    \filldraw[black] (2.5,40) circle (10pt);
    \draw[line width = 0.20mm] (0,30) -- (2.5,40);
    \filldraw[black] (0,50) circle (10pt);
    \draw[line width = 0.20mm] (2.5,40) -- (0,50);
    \filldraw[black] (5,50) circle (10pt);
    \draw[line width = 0.20mm] (2.5,40) -- (5,50);

    \filldraw[black] (12.5,40) circle (10pt);
    \draw[line width = 0.20mm] (15,30) -- (12.5,40);
    \filldraw[black] (10,50) circle (10pt);
    \draw[line width = 0.20mm] (12.5,40) -- (10,50);
    \filldraw[black] (17.5,40) circle (10pt);
    \draw[line width = 0.20mm] (15,30) -- (17.5,40);
    \filldraw[black] (15,50) circle (10pt);
    \draw[line width = 0.20mm] (17.5,40) -- (15,50);
    \filldraw[black] (20,50) circle (10pt);
    \draw[line width = 0.20mm] (17.5,40) -- (20,50);

    \filldraw[black] (27.5,40) circle (10pt);
    \draw[line width = 0.20mm] (30,30) -- (27.5,40);
    \filldraw[black] (25,50) circle (10pt);
    \draw[line width = 0.20mm] (27.5,40) -- (25,50);
    \filldraw[black] (32.5,40) circle (10pt);
    \draw[line width = 0.20mm] (30,30) -- (32.5,40);
    \filldraw[black] (30,50) circle (10pt);
    \draw[line width = 0.20mm] (32.5,40) -- (30,50);
    \filldraw[black] (35,50) circle (10pt);
    \draw[line width = 0.20mm] (32.5,40) -- (35,50);

    \filldraw[black] (42.5,40) circle (10pt);
    \draw[line width = 0.20mm] (45,30) -- (42.5,40);
    \filldraw[black] (40,50) circle (10pt);
    \draw[line width = 0.20mm] (42.5,40) -- (40,50);
    \filldraw[black] (47.5,40) circle (10pt);
    \draw[line width = 0.20mm] (45,30) -- (47.5,40);
    \filldraw[black] (45,50) circle (10pt);
    \draw[line width = 0.20mm] (47.5,40) -- (45,50);
    \filldraw[black] (50,50) circle (10pt);
    \draw[line width = 0.20mm] (47.5,40) -- (50,50);

    \filldraw[black] (57.5,40) circle (10pt);
    \draw[line width = 0.20mm] (60,30) -- (57.5,40);
    \filldraw[black] (55,50) circle (10pt);
    \draw[line width = 0.20mm] (57.5,40) -- (55,50);
    \filldraw[black] (62.5,40) circle (10pt);
    \draw[line width = 0.20mm] (60,30) -- (62.5,40);
    \filldraw[black] (60,50) circle (10pt);
    \draw[line width = 0.20mm] (62.5,40) -- (60,50);
    \filldraw[black] (65,50) circle (10pt);
    \draw[line width = 0.20mm] (62.5,40) -- (65,50);

\end{tikzpicture}
\caption{The branch diagram of a leaf set $B \subseteq 3^5$ with $\LTD(B) = 1$, $\MTD(B) = 2$, and $\TD(B) = 3$.}
\label{fig:strict ineq ha}
\end{figure}

\begin{theorem}\label{thm:ltd ha}
Given a set of leaves $B\subseteq m^n$, if $\LTD_{\ell}(B) \leq d$, then 
\begin{equation}\label{eq:thm ltd ha}
|B| \ \leq \ \sum\limits_{i=0}^d \binom{n}{i}(m-\ell+1)^i(\ell-1)^{n-i}.
\end{equation}
\end{theorem}
\begin{proof}
The theorem clearly holds when $d = -1$ and when $d=n$.

Assume the theorem holds for some $n \geq 0$, and let $B\subseteq m^{n+1}$. Suppose $\LTD_{\ell}(B)\leq d$ with $0\leq d\leq n$. Let $A_1 = \{a\in m^n: a\in \check{B}\}$.
Since $\LTD_{\ell}(A_1)\leq d$, our assumption implies that 
\begin{equation*} \label{eq:thm ltd A_1 ha}
|A_1|\ \leq\ \sum\limits_{i=0}^d \binom{n}{i}(m-\ell+1)^i(\ell-1)^{n-i}.
\end{equation*}
Let $A_2 = \{a\in m^n: |B\dhr_{a}|\geq \ell\}$.
We must have $\LTD_{\ell}(A_2)\leq d-1$, so our assumption implies that 
\begin{align*}\label{eq:thm ltd A_2 ha}
|A_2|\ &\leq\  \sum\limits_{i=0}^{d-1} \binom{n}{i}(m-\ell+1)^i(\ell-1)^{n-i}\\ 
&= \ \sum\limits_{i=1}^{d} \binom{n}{i-1}(m-\ell+1)^{i-1}(\ell-1)^{(n+1)-i}.
\end{align*}
It follows that 
\begin{align*}
|B|&\ \leq \  |A_1|(\ell-1)+|A_2|(m-\ell+1)\\
&\ \leq\  \sum\limits_{i=0}^d \binom{n+1}{i}(m-\ell+1)^i(\ell-1)^{(n+1)-i}
\end{align*}
as desired.
\end{proof}

It is not difficult to see that if $B \subseteq m^n$ is a set of solutions to a hereditary $(d+1)$-fold banned $m$-ary sequence problem of height $n$ (see \cite[Definitions 4.1 and 4.2]{chase:bbsp}), then $\LTD_m(B) \leq d$. Thus, Theorem~\ref{thm:ltd ha} yields \cite[Theorem 4.3]{chase:bbsp} as a corollary.

\begin{cor}
A hereditary $(d+1)$-fold banned $m$-ary sequence problem of length $n$ has at most
\[\sum_{i=0}^d \binom{n}{i} (m-1)^{n-i}\]
solutions.
\end{cor}
 
\subsection{Maximal Leaf Sets}

We say $B\subseteq m^n$ is a \emph{maximal leaf set} with $\ell$-ary tree dimension $d$, and write $\TD_\ell(B) = d^+$, when $\TD_\ell(B) = d$ but $\TD_\ell(B\cup \{b\})> d$ for all $b\in m^n\setminus B$. We follow the same convention for the meeted and leveled variants of $\ell$-ary tree dimension.
 
\begin{definition}
Given a node $a\in m^{\leq n}$, we define its \emph{$\ell$-ary norm} by
\[ \|a\|_{\ell} = \sum_{i<|a|} \begin{cases}
    1 & \text{ if } a(i) \geq \ell-1,\\
    0 & \text{ otherwise}.
\end{cases} \]
\end{definition}

Notice that for any $a \in 2^{\leq n}$, we have $\|a\| = \|a\|_2.$

\begin{definition}
If $A$ is a nonempty set of nodes, we define its \emph{$\ell$-ary norm} by
\[\|A\|_{\ell} = \max_{a\in A} \|a\|_{\ell}.\] 
If $A$ is empty, we set $\|A\|_{\ell} =-1$. 
\end{definition}

It is easy to see that the upper bound provided by (\ref{eq:thm ltd ha}) is tight since $\LTD_\ell(B) = d$ when $B = \{b\in m^n: \|b\|_\ell\leq d\}$.
We will show that (\ref{eq:thm ltd ha}) becomes an equality when $\MTD_\ell(B) = d^+$ (see Corollary~\ref{cor:isotp max td ha}).

\begin{definition}For each $a \in m^{< n}$ and $k < m-1$, let $\sigma_{a, k}: T_{n,m}\rightarrow T_{n,m} $ be the automorphism mapping 
\[a ^\frown k^\frown b \ \mapsto \ a^\frown (k+1) ^\frown b \qquad \text{ and } \qquad a ^\frown (k+1)^\frown b \ \mapsto \ a^\frown k ^\frown b\]
for all sequences $b$ of length at most $n-|a|-1$, leaving all other nodes fixed. 
\end{definition}

\begin{definition}
Given a set of leaves $B\subseteq m^n$, we define its \emph{$\ell$-ary normalization} $\ddot{B}$ using the following inductive procedure:  
    \begin{itemize}[leftmargin=16pt, topsep=4pt, itemsep=4pt]
        \item[] Let $N=m^{n-1}+m^{n-2} +\cdots + m^0$, and let $(a_i:i<N)$ enumerate the elements of $m^{<n}$ in such a way that if $a_i$ is longer than $a_j$, then $i < j.$ Let $B_0 = B$.
        \item[] For each $i<N$, once $B_i$ has been defined: 
        \begin{itemize}[leftmargin=16pt, topsep=4pt, itemsep=4pt]
            \item[] Let $B_{i,0} = B_i$.
            \item[] For each $j< m-1$, once $B_{i,j}$ has been defined:
            \begin{itemize}[leftmargin=16pt, topsep=4pt, itemsep=4pt]
                \item[] Let $B_{i,j,0} = B_{i,j}$.
                \item[] For each $k<m-j-1$, once $B_{i,j,k}$ has been defined:
                \begin{itemize}[leftmargin=16pt, topsep=4pt, itemsep=4pt]
                    \item[] If $\|B_{i,j,k}\dhr_{a_i^{\frown}k}\|_{\ell} - \|a_i^{\frown}k\|_{\ell} \geq \|B_{i,j,k}\dhr_{a_i^{\frown}(k+1)}\|_{\ell}-\|a_i^{\frown}(k+1)\|_{\ell}$,
                    \item[] \hspace{24pt}let $B_{i,j,k+1} = B_{i,j,k}$.
                    \item[] Otherwise, let $B_{i,j,k+1} = \sigma_{a_{i}, k}(B_{i,j,k})$.
                \end{itemize}
                \item[] Let $B_{i,j+1} = B_{i,j,m-j-1}$.
            \end{itemize}
            \item[] Let $B_{i+1} = B_{i,m-1}$.
        \end{itemize}
        \item[] Let $\ddot{B}=B_N$.
    \end{itemize}
\end{definition}

\begin{theorem}
\label{lem:stabless ha}
If $B \subseteq m^n$ is a set of leaves with $\ell$-ary normalization $\ddot{B}$, then $\MTD_{\ell}(B) = \|\ddot{B}\|_{\ell}$. 
\end{theorem}

\begin{proof}
Clearly, we have $\MTD_{\ell}(B) \leq \|\ddot{B}\|_{\ell}$ since $T_{\|\ddot{B}\|_{\ell}+1,\ell}$ cannot embed in \[
(\{a\in m^{\leq n}: \|a\|_{\ell}\leq \|\ddot{B}\|_{\ell}\},\prec).
\] Thus, it is enough to show that  $\MTD_{\ell}(B)\geq d$ whenever $\|\ddot{B}\|_{\ell}\geq d$. This is obviously true when $d\leq 0$. 
Assume it holds for some $d\geq 0$, and suppose we have $b\in \ddot{B}$ with $\|b\|_{\ell} \geq d+1$. Let $a$ be the largest initial segment of $b$ such that $\|a\|_{\ell}=0$. By our assumption, it follows that $T_{d,\ell}$ embeds in $(\check{\ddot{B}}\dhr_{a^\frown k},\prec)$ for each $k< \ell$.
\end{proof}

\begin{cor}\label{cor:isotp max td ha}
Given a set of leaves $B\subseteq m^n$, if $\MTD_{\ell}(B) = d^+$, then \[(\check{B},\prec) \cong (\{a\in m^{\leq n} : \|a\|_{\ell} \leq d\},\prec).\]
\end{cor}

\section*{Appendix. Model-Theoretic Embeddings and Isomorphisms}
In model theory, we say an injective map between two structures is an  \emph{embedding} if it preserves the interpretation of each symbol named in the common language of those structures. If an embedding is bijective, we call it an \emph{isomorphism}. 
%For example, when considering groups, our language has two symbols, the group operation and the identity; thus, embeddings are simply injective homomorphisms. Moreover, in this context, our notion of isomorphism coincides with the standard notion from group theory.  

Whether a map is an embedding depends not only on the way in which it takes elements from domain to codomain but also on the structural context in which we view those domains. In this paper, we consider three kinds of structures, each giving rise to its own notion of embedding. When considering substructures $(A,\prec, \wedge, \sim)$ and $(B,\prec, \wedge, \sim)$ of a leveled tree $\widetilde{T}_k$, we call an injective map $f: A\to B$ an embedding if the following hold for all $a, a'\in A$:
\begin{enumerate}
	\item $a \prec a'$ if and only if $f(a)\prec f(a')$, 
	\item $f(a\wedge a') = f(a) \wedge f(a')$,
	\item $a\sim a'$ if and only if $f(a)\sim f(a')$.
\end{enumerate}
However, to consider the same map an embedding between substructures $(A,\prec)$ and $(B, \prec)$ of the tree $T_k$, we only require (1) to hold. Finally, if both (1) and (2) hold, then $f$ is an embedding between substructures $(A,\prec,\wedge)$ and $(B, \prec,\wedge)$ of the meeted tree $\hat{T}_k$.

For a more detailed review of model-theoretic structures, embeddings, and isomorphisms, see Chapter 1 of \cite{marker:modeltheory}, particularly Definitions 1.1.2 and 1.1.3.

\section*{Acknowledgments} 
The author would like to thank David Marker, Artem Chernikov, and Caroline Terry for helpful comments and suggestions. 

\bibliography{bib}
\bibliographystyle{amsplain}
\end{document}